\documentclass[11pt]{amsart}

\title{A  linear decomposition  attack  }

\author{Alexei Myasnikov}
\address{Department of Mathematical Sciences\\Stevens Institute of Technology}
\curraddr{}
\email{amiasnikov@gmail.com}
\thanks{The first  author was partially supported by  NSF grant
DMS-1318716 and by Russian Research Fund, project 14-11-00085}

\author{Vitali\u\i\ Roman'kov}
\address{Institute of Mathematics and Information Technologies\\Omsk State Dostoevskii University}
\curraddr{}
\email{romankov48@mail.ru}
\thanks{The  second author  was supported by RFBR, project 13-01-00239a and by Russian Research Fund, project 14-11-00085}

\usepackage{amsmath,amsthm,amsfonts,amssymb}
\usepackage{graphicx}
\usepackage{hyperref}
\usepackage[usenames]{color}

\newtheorem{theorem}{Theorem}[section]
\newtheorem{lemma}[theorem]{Lemma}
\theoremstyle{definition}

\newtheorem{corollary}[theorem]{Corollary}

\newcounter{comcount}

\def\MA{{\mathbf{A}}}
\def\F{{\mathbb{F}}}

\date{}

\begin{document}

\maketitle

\begin{abstract}
We discuss  a new attack, termed a {\em dimension} or {\em linear decomposition}  attack,  on  several known  group-based cryptosystems. 
This attack  gives a polynomial time deterministic algorithm that recovers the secret shared key from the public data in all the schemes under consideration. Furthermore, we show that in this case, contrary to the common opinion,  the typical computational security assumptions are not very relevant to the security of the schemes, i.e., one can break the schemes without solving the algorithmic problems on which  the assumptions are based.  The efficacy of the attack depends on the platform group, so it requires a  more thorough analysis in each particular case. 
\end{abstract}

%\tableofcontents

\section{Introduction}
\label{se:intro}

In this paper we discuss, following \cite{Rom3} and \cite{Rom4}, a new general attack on  several known  group-based cryptosystems.  This attack works when the platform  groups (or some other related groups) are linear. We show that in this case, contrary to the common opinion (and some explicitly stated security assumptions), one does not need to solve the underlying algorithmic problems to break the scheme, i.e.,  there is another  algorithm that recovers the private keys without solving the principal algorithmic problem on which the  security assumptions are based. This changes completely our understanding of security of these scheme. The efficacy of the attack depends on the platform group, so it requires a  specific analysis in each particular case. In general one can only state that the attack is in polynomial time in the size of the data, when the platform and related groups are given together with their linear representations. Of course, this requires additional  explanation in the case of finite groups.  

\subsection{ Motivation} 

The purpose of this paper is three-fold. Firstly, we  describe a new very general  attack, called the {\em dimension} or {\em linear decomposition}  attack,  on several known crypto schemes. The attack is based on elementary properties of finite dimensional linear spaces, to apply it one has to represent the  platform  group  (or semigroup, or algebra)  as a subset of  a finite dimensional linear space.  Furthermore, the schemes themselves should be based on  algorithmic problems of a particular type:    the conjugacy search problem,  the decomposition search problem, the factorization problem, automorphic actions,  - those ones that induce linear transformations on the underlying linear space.  Notice, that we do not require here that the group embeds into the linear space homomorphically. Secondly, we show that the security assumptions of the schemes under consideration (either stated explicitly or tacitly assumed in the description of the systems) do not hold in the case of linear groups or semigroups, or algebras. Indeed, we show how one can recover the private keys without solving the underlying algorithmic problems. Thirdly, we discuss how to improve the schemes in order to fail the dimension attack. The easiest way would be  to   choose only non-linear  groups as the platform groups. However, this would leave many interesting groups, in particular all finite groups, out of  reach.  On the other hand,  the group might be linear, but the dimensions of its linear representations could be so large that  the dimension attacks become inefficient. Not much in general is known about efficiency  of linear representations of groups, most of the general problems are wide open.  In particular, it would be very interesting to study minimal (and efficient)  matrix representation of nilpotent and polycyclic groups,  and their holomorphs, and especially finite groups. We discuss this issues below.

\subsection{Results and the structure of the paper}

In Section 2 we describe some typical  cryptosystems  based on non-commutative groups (semigroups, algebras).  Most of them are best understood as various generalizations of the classical Diffie-Hellman (DH) scheme.   In particular, we discuss   schemes based on  the conjugacy search problem, the decomposition and factorization problems, and actions by automorphisms. It worth to mention here that the famous Anshel-Anshel-Goldfeld scheme \cite{AAG} is not part of our discussion, since, on the one hand,  it is not of the Diffie-Hellman type, and on the other hand, it is currently  under another powerful linear algebra attack from \cite{Tsaban}.

In Section 3 we explain the basic ideas of the linear decomposition attacks and provide some useful algorithms. The algorithms work in  polynomial time in the size of the public data described in the schemes. All these results are rather theoretical: we did not try to improve  neither on the efficacy of the algorithms nor on the complexity bounds, no doubts that both could be tighten up considerably.  

In Section 4 we give cryptanalysis of a variety of group-based protocols of the Diffie-Hellman type.  There are two main results here.  The first one claims that  the linear decomposition attacks give a polynomial time deterministic algorithms that recover the secret shared key from the public data in all the schemes we discuss in the section. The second one shows that in all these schemes the typical computational security assumptions are not very relevant to the security of the schemes, i.e., one can break the schemes without solving the algorithmic problems on which the  assumptions are based.  This is rather striking since the schemes together with the assumptions were known for some time and studied quite thoroughly.

\subsection{Efficacy of linear decomposition attacks}

As we have mentioned above  some of the DH-like schemes  based on group theoretic problems are susceptible to linear decomposition attacks, provided that the platform groups (or some other related groups) are linear. One of the possible ways to fail the attack is to use non-linear groups as platforms. However, this would leave behind all finite groups, which are known to be  linear.  Another way is to allow linear groups $G$ as platforms, but only those ones whose   smallest faithful linear representations $G \to GL_n(\F)$ have "prohibitively  high" dimensions $n$ that makes dimension attacks implausible. Of course, if the platform group $G$ is fixed  then the dimension of any fixed linear representation of $G$ is  a constant, so the dimension attack still works in polynomial time (though the constants in the polynomials could be very large).  To make the situation more clear one has to consider a class $\mathcal C$ of platform groups such that the dimension $n(G)$ of a smallest linear representation of $G$ grows exponentially with respect to the size of the given description of $G$ (for example,  with respect to the size of the given finite presentation of $G$ in terms of generators and relators). This approach accommodates naturally any infinite class $\mathcal C$ of finite groups. 

Notice, that the class of linear groups is very large. Besides finite groups it contains finitely generated nilpotent groups, arbitrary polycyclic groups, right angled Artin and braid groups, holomorphs of polycyclic groups, etc. Not all finitely generated metabelian groups are linear, but all of them admit faithful representations into finite direct products of general linear groups. Surprisingly,  not much is known about the dimensions of smallest faithful linear representations of such infinite groups as nilpotent, or polycyclic, even less about  metabelian groups.  
Graaf and Nickel \cite{GN} and  Nickel in \cite{N} studied the classical linear representations of finitely generated torsion-free nilpotent groups $G$ into $UT_n(\mathbb{Z})$ from the algorithmic view-point, showing in particular that there are polynomial time (in the Hirsh length of $G$) algorithms to construct the representations.  Recently, in \cite{HK}  Habeeb and Kahrobaei, following \cite{N}, showed that the dimension of the  classical representations above is $O(n^2)$ where $n$ is the Hirsh length of the group.   It seems  the dimension is not prohibitively high in this case. However, there nilpotent groups with presentations (in generators and relators)  which size is logarithmically smaller then their Hirsh length. Indeed, free nilpotent groups of class $c$ and rank $r$ have presentations with $r$ generators and no relations in the variety of nilpotent groups of class at most $c$, but their Hirsh length is about $r^{c+1}$. On the other hand, to use a group $G$ as a platform in cryptography  requires a fast solution of the word problem in $G$, as well as fast algorithms for computing normal forms of group elements (see \cite{MSU2} for details). It seems the known algorithms for computing normal forms in nilpotent groups rely on so-called polycyclic presentations of nilpotent groups, which size is about the same as the Hirsh length of the group.  Situation with finite nilpotent groups is even more interesting. Some upper bounds for dimensions of smallest linear representations of say finite $p$-groups are known: in \cite{J} Janusz showed that the minimal faithful representation of a finite p-group as a group of matrices over a finite field of characteristic $p$ is $1 + p^{(e - 1)}$ where $e$ is the exponent of $G$, so the degree here is quite large (for some groups it is comparable with the order of the group). 

The discussion above  highlights several interesting open problems in algorithmic group theory, whose solution would shed some light on security of the corresponding cryptosystems.

\section{Crypto schemes under the dimension attack}
\label{se:schemes}

Now we describe  cryptosystems which to some extent are susceptible to  the dimension attack. But first  a few words on terminology. Most of the crypto schemes (systems, protocols) discussed below were not originally stated within the formal rules of the current cryptographic practice (see, for example,  the book \cite{KL} for definitions of a cryptosystem)  - usually some of the required algorithms or parameters  are not completely described, or security assumptions are missing.  In particular, it is hard to break such schemes because the precise description is lacking. However, they make perfect sense as general ideas or "general schemes" from which the concrete protocols should be worked out after some research on sorting out which parameters  are strong and which are not.  The research itself usually comes as a series of "attacks" on the scheme and the subsequent cryptanalysis. Our main intention in designing the dimension attack is not on breaking but on improving the generic schemes at hands. In what follows we are focusing on the choice of the so-called {\em platform group}, one of the main parameters in the group-based cryptosystems, and discuss how security of the scheme depends on the chosen platform.  In particular, we shed some light on some known (computational) security assumptions.

Most of the schemes we discuss in this paper the best can be seen  in the light of the  famous Diffie-Hellman (DH) key establishment protocol \cite{DH}.  The main idea of DH  is very simple and can be described  as follows. Two users, say Alice and Bob, first choose the multiplicative group of integers $G = \mathbb{Z}^{\ast}_p$ modulo a prime number $p$, as the platform group,  and some element $g \in G$  (all this data is public).  Alice then  selects  a random   integer $k \in \mathbb{N}$ (her private key)  and  sends  $g^k$  to Bob through a public channel. He  in turn picks  $l \in \mathbb{N}$  (his private key) and sends  $g^l$   to Alice through a public channel. Both Alice  and Bob  can then
compute their secret shared keys  $K = g^{kl}$. 
An  adversary, say Eve,  monitoring the transmission between Alice  and Bob knows the public data $G, g, g^k, g^l$  and her task is to recover the shared key   $K$ (i.e., to break the scheme).  The {\em computational security assumption}  claims that  it is a time consuming task to recover the shared key from the public data. Namely,  the claim is that for a fixed $G$, any probabilistic polynomial-time algorithm succeeds in breaking the scheme with only negligible probability. We refer to \cite{KL} for precise formulations.   Observe, that  if Eve  could compute either $k$  or $l,$ then she  could find $K$ easily. Thus we arrive to the {\em underlying algorithmic problem}: recover Alice's (or Bob's) private key from the public data. In the case of Diffie-Hellman the underlying algorithmic problem is the famous {\it  discrete logarithm problem}   for $G$, which asks whether or not one can compute (minimal non-negative) $k$ from given $g$ and $g^k$ in probabilistic Ptime. For other schemes the underlying algorithmic problems could be  different.  Notice, that it is important how the group $G$ and its elements are given: algebraically the same cyclic group of order $p$ (and its elements) could be given as integers between $0$ and $p-1$, or as the group of an elliptic curve, or  by a finite presentation, - the scheme security depends on the presentation. Furthermore, the scheme still makes sense when the group $\mathbb{Z}^{\ast}_p$ is replaced by an arbitrary finite (or infinite) group $G$. This gives rise to a general  DH scheme. Various  platform groups $G$ were suggested and studied, finite or not:  the group of non-singular matrices over a finite field \cite{COS}, \cite{OVS},  over a group algebra \cite{KKS}, over a semigroup \cite{GS}, etc.

Now we discuss  some group based cryptoschemes with respect to the  algorithmic problems they are based on.

\medskip \noindent
{\bf Schemes based on the conjugacy search problem}.
One of the possible generalizations of the general DH scheme to arbitrary non-commutative group is to use the conjugacy in the place of exponentiation.  Recall, that the  conjugate  of an element $g$ by an element $x$ in a group  is defined by $g^x = xgx^{-1}.$  The map $\phi_x:g \to g^x$ is an automorphism of $G$, called {\em conjugation}.
One of the principal  DH-type schemes based on the conjugation (instead of exponentiation) was introduced by  Ko, Lee et. al. in  \cite{KLCHKP}.  In this case let $G $ be a platform group and $U, W$ be two finite subsets of $G$ which are commuting element-wise, i.e., $uw = wu$ for any elements $u \in U, w \in W$. Denote by  $A$ and $B$ the  subgroups of $G$ generated by $U$ and $W$ correspondingly and fix an element $g \in G$. All this data is assumed to be public.  Then  Alice picks a private element $a \in A$  and publishes   $g^a$. Bob picks a private element $b \in B$ and  publishes  $g^b.$ After that 
Alice computes  a shared secret $K_A = (g^b)^a = g^{ab}$, while  Bob computes  the same element as $K_b = (g^a)^b = g^{ba} = g^{ab}.$  In this scheme the following is the underlying algorithmic problem (like the discrete log for DH scheme):
 
\begin{itemize}
\item {\em The conjugacy search  problem (CSP) in a group $G$:  given two elements 
$g, f \in G$ and information that $g^x = f$ for some $x \in G,$ find at least one particular element $x$ like that.}
\end{itemize}

The CSP plays a special role in group-based cryptography. Many  protocols  and cryptosystems based on groups use one or another variation  of CSP.  For example, the schemes  \cite{KK, KLCHKP, SU3,WWCOS,GS}   use  CSP; \cite{AAG, AAGL} use  the simultaneous CSP (when one has to solve in a group $G$ a system of the type $g_1^x = f_1, \ldots, g_k^x = f_k$).

\medskip \noindent
{\bf Schemes based on the decomposition and factorization problems}. Again,  DH scheme in a group $G$ can be simulated by replacing exponentiation by right and left multiplication. One of the most typical schemes in this area is due to Shpilrain and Ushakov  \cite{SU4}.  In the notation above the idea of the scheme is as follows. A group $G$, two element-wise commuting subgroups $A$ and $B$, and a fixed  element  $g \in G$ are given. Then Alice picks private elements  $a, a' \in A$ and publishes  the element $aga'.$  Bob picks private elements $b, b' \in B$  and  publishes the element $bgb'.$ After that 
Alice computes  a shared secret key  $K_A = abgb'a'$, while  Bob computes  the same element as $K_B = baga'b' = abgb'a'.$ In this scheme the underlying algorithmic problem is the decomposition search problem.

\begin{itemize}
\item 
{\em The decomposition search problem (DSP) in a group $G$:  given two subgroups $A,B \subseteq G$ and two elements $g, f \in G, $ find elements $a \in A$ and $b \in B$ such that  $a\cdot g \cdot b = f,$  provided  that at least one such pair of elements exists.}
\end{itemize}

There are two  variations of DSP, which have been used in group-based crypto:

\begin{itemize}
\item
{\em The  factorization search problem (FSP) in a group $G$: given an element $f \in G$ and two subsets (usually subgroups)  $A$ and $B$ of $G$, find elements $a \in A$ and $b \in B$ such that $a \cdot b = f.$}

\item 
{\em The  power conjugacy search problem (PCSP)  in a group $G$: given two elements 
$g, f \in G$ and information that $(g^k)^x = f$ for some $k \in \mathbb{N}$ and $x \in G,$ find at least one particular pair $(k, x)$ like that. }
\end{itemize}

Here are some schemes based on DSP and FSP:
\cite{AMVZ}, \cite{ATVZ1}, \cite{ATVZ2},  \cite{RU}, \cite{Shp}, \cite {SU4},    \cite{SU2},  \cite{SCY},  \cite{Sti}. The schemes \cite{KK}, \cite{STR} use PCSP  in matrix groups.

\medskip \noindent
{\bf Schemes using actions by automorphisms}. One natural generalization of DH scheme is to replace exponentiation by arbitrary commuting automorphisms of the group $G$. When the automorphisms are conjugations one gets the scheme \cite{KLCHKP} described above. More precisely, let $G$ be a group   and $g \in G$. By $Aut(G)$ we denote the group of automorphisms of $G$. For $g \in G$ and $\phi \in Aut(G)$ by $g^\phi$ we denote the image of $g$ under $\phi$.  Suppose $U, W$ be two finite subsets of  $Aut(G)$ commuting element-wise. Denote by $A$ and $B$ the subgroups in $Aut(G)$ generated by $U$ and $W$ respectively.  Now 
Alice picks $a \in  A $ and publishes $g^a.$ Bob picks $b \in B$ and publishes  $g^b.$ Then  Alice computes  a secret shared key $K_A = (g^b)^a = g^{ba}$, while  Bob computes the same element as $K_B = (g^a)^b = g^{ab} = g^{ba}$.  The underlying algorithmic problem is the search automorphism problem:

\begin{itemize}
\item  {\em The search automorphism problem in $G$: given a subgroup $A \leq Aut(G)$ and two elements $g,h \in G$ find an automorphism $a \in A$ such that $g^a = h$, provided that such an automorphism exists. }
\end{itemize}

Schemes from \cite{ER}, \cite{GS},  \cite{Mah} use the search automorphism (or endomorphism) problems in their design.

\section{The principle idea}
\label{se:bil}

In this section we describe the mathematical idea behind the linear decomposition attacks. Our exposition is closely linked to the "prototypical" schemes discussed in Section \ref{se:schemes}. 

\subsection{Finding a basis}

Let $V$ be a finite dimensional vector space over a field $\mathbb{F}$  with basis $\mathcal{B} = \{v_1, \ldots,v_r\}$.  Let $End(V)$ be  the  semigroup of endomorphisms of $V.$  We assume that  elements $v \in V$ are given as vectors relative to $\mathcal{B}$, and endomorphisms $a \in End(V)$ are given by their matrices relative to $\mathcal{B}$. 
For an endomorphism $a \in End(V)$ and an  element $v \in V$ we denote by  $v^a$ the image of $v$ under $a.$  Also, for any subsets $W \subseteq V$ and $A \subseteq End(V)$ we put $W^A = \{w^a | w \in W, a \in A\}$, and  denote by $Sp(W)$ the subspace of $V$ generated by $W$, and by   $\langle A \rangle$ the submonoid generated by $A$ in $End(V)$.

The discussion below concerns with time complexity of some algorithms. To this end we put some assumptions on computations in $\F$.

\medskip \noindent
{\bf Computational assumption on the fields } {\it We assume that elements of the field $\F$ are given in some constructive form and the "size" of the form is defined. Furthermore, we assume that the basic field operations in $\F$ are efficient, in particular they can be performed in polynomial time in the size of the elements. In all the particular protocols considered in this paper  the filed $\F$ satisfies all  these conditions. }

\medskip
For an element $\alpha \in \F$ we write $\Vert \alpha \Vert$ for the size of 
$\alpha$ and put $\Vert v \Vert = \max{\Vert \alpha_i\Vert}$ for a vector $v = (\alpha_1, \ldots, \alpha_r) \in V$, and  $\Vert a\Vert  = \max\{\Vert\alpha_{ij}\Vert \}$ for a matrix $a= (\alpha_{ij}) \in End(V)$. 

\begin{lemma} [Principal Lemma] \label{le:principal}
There is an algorithm that for given finite subsets $W \subseteq V$ and $U  \subseteq End(V)$ finds  a basis of the subspace $Sp(W^{\langle U\rangle})$   in the form $w_1^{a_1}, \ldots, w_t^{a_t}$, where $w_i  \in W$ and $a_i$ is a  product  of elements  from $U$.  Furthermore, the number of field operations used by the algorithm is polynomial in $r = \dim_\F V$ and the cardinalities of  $W$ and $U$.
\end{lemma}
\begin{proof}
Using Gauss elimination one can effectively find  a maximal linearly independent subset $L_0$ of $W$. Notice that $Sp(L_0^{\langle U\rangle}) = Sp(W^{\langle U\rangle})$. Adding to the set $L_0$  one by one elements $v^a$, where $v \in L_0, a \in U$ and checking every time  linear independence of the extended set, one can effectively construct a maximal linearly independent subset $L_1$ of the set $L_0 \cup L_0^U$ which extends the set $L_0$. Notice that $Sp(L_0^{\langle U\rangle}) = Sp(L_1^{\langle U\rangle})$ and the elements in $L_1$ are of the form $w^a$, where $w \in W$ and $a \in \langle U\rangle$. It follows that if $L_0 = L_1$ then $L_0$ is a basis of $Sp(W^{\langle U\rangle})$. If $L_0 \neq L_1$ then we repeat the procedure for $L_1$ and find a maximal linearly independent subset $L_2$ of $L_1 \cup L_1^U$ extending $L_1$. Keep going one constructs a sequence of strictly increasing subspaces $L_0 < L_1 < \ldots < L_i$ of $V$. Since the dimension $r$ of $V$is finite  the sequence stabilizes for some $i \leq r$. In this case  $L_i$ is a basis of  $Sp(W^{\langle U\rangle})$ and its elements are in the required form. 

To estimate the upper bound of the number of the field operations used by the algorithm, observe first that the number of the field operations in Gauss elimination performed on  a matrix of size $n\times r$ is $O(n^2r)$. Hence it requires at most $O(n^2r)$ steps to construct $L_0$ from $W$, where $n = |W|$ is the number of elements in $W$. Notice that $|L_j| \leq r$ for every $j$. So to find $L+{j+1}$ it suffices to perform Gauss elimination on the matrix corresponding to $L_j \cup L_j^U$ which has size at most $r +r|U|$.
Thus the upper estimate on this number is $O(r^3|U|^2)$. Since there are at most $r$ iterations of this procedure one has the total estimate as $O(r^3|U|^2 +r|W|^2)$. Of course, this estimate is very crude.
\end{proof}

\begin{corollary}
With our assumptions on the field $\F$ the algorithm in Lemma \ref{le:principal} works in polynomial time in the size of the inputs, i.e., in $r = \dim_\F V$,  $|W|$, $|U|$, and $\max\{\Vert w\Vert, \Vert u\Vert \mid w \in W, u \in U\}$.
\end{corollary}
%%%%%%%%%%%

Notice, that the algorithm just described in lemma \ref{le:principal} can be obviously adapted to noetherian modules over commutative rings.  

\subsection{The basic linear decomposition attack}
\label{se:basic}

Let as above  $V$ be a finite dimensional vector space over a field $\mathbb{F}$  with basis $\mathcal{B} = \{v_1, \ldots,v_r\}$ and   $U$ and $W$ be  two finite subsets of $End(V)$. 

\medskip \noindent
{\bf Commutativity assumption} {\it We assume that every element of $U$ commutes with every element of $W$, i.e., for every $v \in V, u \in U, w \in W$ one has $v^{uw} = v^{wu}$. }

\medskip \noindent
 Let $A$ and $B$ be the submonoids of $End(V)$ generated by $U$ and $W$ correspondingly.   Suppose  that $a \in A, b \in B$ and $v \in V$. We assume that the field $\F$, the space $V$, the sets $U, W$ and the vectors $v, v^a, v^b$ are public, while the endomorphisms  $a$ and $b$ are private. By the size of the public data we mean the total size of the following parameters: $r = \dim_\F V$ (given in unary, i.e., as $1^r$),  $|W|$, $|U|$, $\max\{\Vert w\Vert, \Vert u\Vert \mid w \in W, u \in U\}$, $\Vert v\Vert,  \Vert v^a\Vert,  \Vert v^b\Vert$.

\medskip
{\bf Claim 1.} {\it Given $U, W, v, v^a, v^b$ one can find in polynomial time (in the size of the public data) the vector $v^{ab} = v^{ba}$.}

\begin{proof} Indeed, given $U$ and $v$ by Principal Lemma (and its corollary) one can find in polynomial time a basis of $Sp(v^A)$ in the form $v^{a_1}, \ldots, v^{a_t}$, where $a_i \in A$ given as some particular products of elements from $U$. Using Gauss elimination one can decompose $v^a$ as a linear combination in the given basis:
$$
v^a = \Sigma_{i = 1}^t \alpha_i v^{a_i}, \ \ \ \alpha_i \in \F.
$$
This allows one to compute $v^{ab}$ as follows:
$$
v^{ab} = (v^a)^b = (\Sigma_{i = 1}^t \alpha_i v^{a_i})^b = \Sigma_{i = 1}^t \alpha_i v^{a_ib} = \Sigma_{i = 1}^t \alpha_i v^{ba_i} = \Sigma_{i = 1}^t \alpha_i (v^b)^{a_i}.
$$
which is immediate, since the vector $v^b$ and the matrices $a_i$ are known. 

The conclusion is that one does not need to find neither $a$ nor $b$ to compute the vector $v^{ab}$.
\end{proof}

\subsection{A linear group acting by  conjugation}
\label{se:linear}

Let $G$ be a finitely generated group that comes equipped with an injective homomorphism $\phi:G \to GL_n(\MA)$, where $\MA$ is a finite dimensional associative algebra over a field $\F$. 

\medskip \noindent
{\bf Computational assumption on $\MA$} {\it As usual we assume that elements of $\MA$ are given in some constructive form and the "size" of the form is defined. Furthermore, we assume that the basic algebra  operations in $\MA$ are efficient, so  they can be performed in polynomial time in the size of the elements. In particular, the matrix multiplication in $Mat_n(\MA)$ can be performed in polynomial time.  Of course, all the conditions above obviously hold in the case when $\MA$ is just the field $\F$.}

  \medskip
  Since the group $G$ is finitely generated and multiplication  in $GL_n(\MA)$ is efficient one can compute in polynomial time the image $g^\phi$ for any element $g \in G$, given as a word in a fixed finite set of generators of $G$. 
  
  Notice, that $V = Mat_n(\MA)$ can be viewed as a finite dimensional vector space over $\F$, where matrices from $Mat_n(\MA)$ are tuples  of length $n^2$ over $\MA$, i.e., elements from $\MA^{n^2}$. If $\MA$ has dimension $r$ over $\F$ then $\MA^{n^2}$ can be viewed as a vector space over $\F$ of dimension $rn^2$ in which addition naturally comes  from  the matrix addition in $Mat_n(\MA)$. The group $GL_n(\MA)$ acts on $V$ by left as well as right multiplication. In both cases the homomorphism $\phi$ gives a faithful representation $\phi:G \to End(V)$. It follows that any two given elements $g, h \in G$ determine an endomorphism  $E_{g,h}:  Mat_n(\MA) \to Mat_n(\MA)$ defined by  $E_{g,h}(x) =  \phi(g)x\phi(h)$.
  In particular, the conjugation by $g \in G$ in $Mat_n(\MA)$ corresponds to the endomorphism $E_{g,g^{-1}}$.

Let  $U$ and $W$ be  two finite subsets of $G$ satisfying the {\it commutativity assumption} as above, i.e.,  every element of $U$ commutes with every element of $W$.  Let $A$ and $B$ be the submonoids of $G$ generated by $U$ and $W$ correspondingly.  Suppose that $a \in A$,  $b \in B$, and $v \in G$.  Put $v^a = ava^{-1}, v^b = bvb^{-1}$.  

We also assume that  the algebra $\MA$, the group $G$, the embedding $\phi:G \to GL_n(\MA)$, the sets $U$ and $W$, as well as the elements $v, v^a, v^b$ are public. As above the size of the public data is the total size of the following parameters: $r$ and $n$ (given in unary), the sizes of $U$ and $W$, and the sizes of all public elements. Notice, that in this case we may assume that elements of $G$ are given as words in a fixed finite generating set. Due to our assumptions on $\MA$ and $G$ the embedding $\phi$ is computable in polynomial time in the length of the words representing elements in $G$, therefore  the sizes of elements of $G$ computed as lengths of the words or the norms of the corresponding matrices are within the polynomial bounds of each other. This implies that the time complexity estimates for our algorithms will be similar if we use representations of the elements as words or as the corresponding matrices.

\medskip
{\bf Claim 2.} {\it Given $U, W, g, g^a, g^b$ one can find in polynomial time (in the size of the public data) the element $g^{ab} = g^{ba}$.}

\begin{proof}
Indeed, the argument above shows that the embedding $\phi:G \to Mat_n(\MA)$ gives, in fact, an embedding $\phi:G \to End(V)$ in such a way that conjugation by an element $g \in G$ gives rise to an endomorphism $E_{g,g^{-1}} \in End(V)$. Since the embedding $\phi$ is polynomial time computable one  finds himself in the situation of the basic linear decomposition attack. Now Claim 2 follows immediately from Claim 1. 
\end{proof}

There are several possible variations or generalizations of the basic scheme described in this section, which also could be easily reduced to the basic model. We mention  some of them below.  

\subsection{ A linear group acting by right/left multiplication}
\label{se:right/left}

We assume all the notation from Section \ref{se:linear}.
Beyond that assume also that $a,a^\prime \in A$ and $b, b^\prime \in B$.  

\medskip
{\bf Claim 3.} {\it Given $U, W, g, agb, a^\prime gb^\prime$ one can find in polynomial time (in the size of the public data) the element $a^\prime a gbb^\prime = aa^\prime gb^\prime b$.}

\begin{proof}
Indeed, an argument similar to the one in Section \ref{se:linear}  reduces Claim 3 to the basic model. 
\end{proof}

Obviously, Claim 2 is just a particular case of Claim 3. Observe also, that Claim 3  holds if one replaces a group $G$ by a semigroup   $G$. The same argument works in this case as well.

\subsection{Groups acting by automorphisms}
\label{se:auto}

Let $G$ be a finitely generated group and $Aut(G)$ the group of automorphisms of $G$. Let  $U$ and $W$  be  two    finite subsets of $Aut(G)$ satisfying the commutativity assumption as above:  every element $u$ of $U$ commutes with every element $w$ of $W$ in $Aut(G)$,  i.e., for any $g \in G$ the equality $ g^{uw} = g^{wu}$ holds.

Denote by $A$ and $B$  the submonoids (or subgroups)  of $Aut(G)$ generated by $U$ and $W$ correspondingly. Now one can consider an analog of the situation described in Section \ref{se:linear}, where the  conjugations are replaced by arbitrary automorphisms. Namely, suppose some automorphisms $a \in A, b \in B$ are chosen. The question arises weather there is a polynomial time algorithm which when given $U \subseteq Aut(G), W \subseteq Aut(G)$, $g \in G$,  and the images $g^a$ and $g^b$ for some elements $a \in A, b \in B$ computes the element  $g^{ab} = g^{ba}$ in $G$. Even if the group $G$ is linear (and the embedding $\phi:G \to GL_n(\MA)$ is given) still in this case there is no obvious reduction to Claims 1 or 2. Indeed, in this case  arbitrary automorphisms from $Aut(G)$ do not in general induce  endomorphisms on the linear space $V$ (in the  notation above).  However the reduction would be possible if one can interprete the automorphisms as conjugations in some (perhaps larger) linear group.  Now we discuss one group theoretic construction that can be useful here.

Recall that the {\em holomorph} $H(G)$ of a group $G$ is a semidirect product $H(G) = G \rtimes Aut(G)$ of $G$ and $Aut(G)$, where the multiplication  on pairs from $G \times Aut(G)$ is defined by $(g,a)(h,b) = (gh^a,ab)$. By construction the groups $G$ and $Aut(G)$ embed into $H(G)$ via injections $g \to (g,1)$ and $a \to (1,a)$. Notice, that every automorphism $a \in Aut(G)$ acts on $G$ by a conjugation in $H(G)$, since $(1,a)(h,1)(1,a^{-1}) = (h^a,1)$.   It follows that if the holomorph $H(G)$ is a linear group, in particular if there is an  injective homomorphism $\phi:H(G) \to GL_n(\MA)$ for some $n$ and $\MA$ as above, then this case can be reduced to Claim 2. In particular, the following result holds.

\medskip
{\bf Claim 4.} {\it Suppose that $H(G)$ is a linear group. Then in the notation above given $U, W, g, g^a, g^b$ one can find in polynomial time (in the size of the public data) the element $(g^a)^b = (g^b)^a$. }

\medskip
Observe, that the holomorph $H(G)$ is linear when the group  $G$ is finite or polycyclic-by-finite \cite{Mer}.

Suppose, a platform group $G$ = gp$(g_1, ..., g_r)$ is given by its generators and defining relations. To apply the linear decomposition we need in  effective embedding $\mu $ of $G$ into a linear group GL$_n(\mathbb{F}).$  Let $\mu (g)$ be a secret date that we get applying our approach to $\mu (G).$ We have to recover $g$ as the result in the original language. Suppose, there is an effective procedure of rewriting $\mu (g)$ as a word $w(\mu (g_1), ..., \mu (g_r))$ in the images $\mu (g_i), i = 1, ..., r,$ of generators. Then $g = w(g_1, ..., g_r)$ and we succeed. There are several papers giving such algorithms. See \cite{BBS} and references there. Let us to cite from this paper: "A constructive membership  test not only answers the question whether or not a given element belongs to a given group but in the case of positive answer, it also provides a straight-line program that constructs the given element from the given generators of the group."

Recall, that constructive membership is the problem of expressing an element in terms of the generators of the group.

 Let $G$ be a group and $S \subseteq G.$ A straightline program reaching some $g \in G$ from $S$ is a sequence $(w_1, ..., w_m), w_i \in G,$ such that for each $i$ either $w_i\in S$ or $w_i = w_j^{-1}$ for some $j < i$ or $w_i =w_jw_k$ for some $j, k < i.$ 

Let $G \leq H$ be groups; let $G$ be given by a generating set $S.$ The constructive membership problem for $G$ in $H$ is, given $g \in H,$ decide whether $g \in G,$ and if so find a straight-line program over $S$ reaching $g.$

 There is a randomized polynomial-time algorithm which uses number theory oracles and given a matrix group $G$ of odd characteristic $p$ solves constructive membership in $G.$ 

Previously similar results were given by E.M. Luks \cite{Luks}  for solvable matrix groups only.    Luk's algorithms are deterministic. Other algorithms and their analysis build on a large body of prior work and most notably on the papers \cite{BPS}, \cite{PW}, \cite{HLBRW}.

\section{Cryptanalysis of protocols}
\label{se:bil}

\subsection{Protocols based on conjugation}

\medskip
\paragraph{1)  Ko, Lee et. al. key establishment protocol} \cite{KLCHKP}.

\medskip
Let $G $ be a group and $U, W$ be two finite subsets of $G$ which are commuting element-wise. Denote by  $A$ and $B$ the  subgroups of $G$ generated by $U$ and $W$ correspondingly. Fix an element $g \in G$.   
We assume that all the data above is public. 

\medskip
{\it Algorithm.} Alice picks a private element $a \in A$  and publishes  $g^a$. Bob picks a private element $b \in B$ and  publishes  $g^b.$

\medskip
{\it  Key establishment.}  
Alice computes $K_A = (g^b)^a = g^{ab}.$ Bob computes  $K_b = (g^a)^b = g^{ba} = g^{ab}.$ The shared key is $K = K_A = K_B = g^{ab}.$ 

\medskip
{\it  Cryptanalysis.} If the group $G$ is linear then by Claim 2 there exists an algorithm that given the public data above finds the shared key $K$ in polynomial time.

In the original version of this cryptosystem  \cite{KLCHKP}    $G$ was proposed to be the Artin braid group $B_n$ on $n$ strings.  R. Lawrence described in 1990 a family of so called Lawrence representations of $B_n.$ Around 2001 S. Bigelow \cite{Big} and D. Krammer \cite{Kr} independently proved that all braid groups $B_n$ are linear. Their work used the  Lawrence-Krammer representations $\rho_n : B_n \rightarrow GL_{n(n-1)/2}(\mathbb{Z}[t^{\pm 1}, s^{\pm 1}])$ that   has been proved faithful for every $n\in \mathbb{N}.$  One can effectively find the image $\rho_n (g)$ for every element $g \in B_n.$ 
Moreover, there exists an effective procedure to recover a braid $g \in B_n$ from its image $\rho_n (g).$ It was shown by J.H. Cheon and B. Jun in \cite{CJ} that it can be done in $O(2m^3log d_t)$ multiplications of entries in $\rho_n (g).$ Here $m = n(n-1)/2$ and $d_t$ is a parameter  that can be effectively computed by $\rho_n (g).$  See \cite{CJ} for details.   Therefore, in this case   there is a polynomial time algorithm to find the shared key $K$ from the public data. The algorithm presented here is more practical. Constructing of a basis is off-line. In every session we have to find  on-line by the Gauss elimination process coordinates of elements in given basis of  vector space.  

Let $Alg(CJ)$ denotes the Cheon-Jun's algorithm, and let $Alg(LD)$ denotes the linear decomposition attack. We can compare these two algorithms. 

\begin{enumerate}
\item  $Alg(CJ)$ is not deterministic because it  look for invertible solutions of underlying sets of linear equations. $Alg(LD)$  is completely deterministic. 
\item   $Alg(CJ)$ works on-line.      $Alg(LD)$ works mostly off-line.
\item  $Alg(CJ)$ deals with bigger sets of variables and equations than $Alg(LD)$ does.
\item $Alg(CJ)$ uses specific embedding  $\mu $,  but $Alg(LD)$ can work with all effective linear representations. Thus $Alg(LD)$ can be used on different platforms. 
\end{enumerate}

\medskip
\paragraph{2) Wang, Cao et. al. key establishment protocol} \cite{WWCOS}.

\medskip
Let $G $ be a  non-commutative monoid. Fix an element $g \in G$.   Let $x$ be an invertible element of $G.$ It is assumed that $G, g, x$ are public. 

\medskip
{\it Algorithm.} Alice picks a private number $s \in \mathbb{N}$ and publishes  $g^{x^s}.$ Bob picks a private number $t \in \mathbb{N}$ and   publishes $g^{x^t}.$

\medskip
{\it Key establishment.} Alice  computes  $K_A = (g^{x^t})^{x^s}= g^{x^{s+t}}.$     Bob computes  $K_B = (g^{x^s})^{x^t} = g^{x^{s+t}}.$
The shared key is  $K = K_A = K_b = g^{x^{s+t}}.$

\medskip
{\it  Cryptanalysis.} If the monoid $G$ is linear then by Claim 2 there exists an algorithm that given the public data above finds the shared key $K$ in polynomial time.

However, in the paper  \cite{WWCOS} the authors used the semigroup $G$ of $3 \times 3$ matrices of $1000$-truncated polynomials in $10$ variables over the ring (not a field) $\mathbb{Z}_{12}$, which does not reduces directly to Claim 2. 

Nevertheless, a slight modification of the linear decomposition attack works in this case as well.

\subsection{Protocols based on left/right multiplication}

\medskip
\paragraph{3) B. Hurley and T. Hurley's  authentication and digital signature  protocols} \cite{HH}, \cite{H}. 

\medskip

Let $G$ be a commutative subgroup of  GL$_n(\mathbb{F}).$ These dates are public.

{\large Algorithm:}

\begin{enumerate}
\item Bob picks   $y \in \mathbb{F}^n$ and  $B \in G$, computes and publishes  $yB$.
\item    Alive wants to send a message $x \in \mathbb{F}^n $ to Bob. She picks  $A_1, A \in G,$  computes and sends $(xA, yBA_1)$ to Bob.
\item Bob picks $B_1, B_2 \in G$, computes and sends  $(xAB_1, yA_1B_2)$ to Alice.
\item Alice computes  $(xB_1, yB_2)$ and sends  $xB_1-yB_2$ to Bob.
\item  Bob computes  $x - yB_2B_1^{-1} $ and recovers $x.$
\end{enumerate}

Bob may use  $yB$ in further transactions.

\medskip
{\it  Cryptanalysis.} Since  the group $G$ is linear then by Claim 3 there exists an algorithm that given the public data above finds the message $x$ in polynomial time. Indeed, let us describe the recovering algorithm.

\begin{enumerate}
\item  By Claim 3 we build a basis of the space Sp$(yBA_1)G.$ Let this basis is   $\{ yBA_1C_1, ..., yBA_1C_r\},$ where $C_i \in G, i = 1, ..., r.$
\item Then we obtain  $yB = \sum_{i=1}^r\alpha_iyBA_1C_i, \alpha_i \in \mathbb{F}.$
\item Swap  $yBA_1$  by $yA_1B_2.$ We have  $\sum_{i=1}^r\alpha_iyA_1B_2C_i = yB_2.$
\item Similarly, we construct a basis  
$xAB_1D_1, ..., xAB_1D_t, D_j \in G, j =1, ..., t,$ of  Sp$(xAB_1)G.$ 
\item Compute $xA = \sum_{i=1}^t\beta_ixAB_1D_i, \beta_i \in \mathbb{F}.$
\item Swap $xAB_1$  by $xB_1 - yB_2.$ We have $\sum_{i=1}^t\beta_i(xB_1-yB_2)D_i = x - yB_2B_1^{-1}.$
\item Swap again  $xAB_1$ by $yB_2$ and get $\sum_{i=1}^t\beta_iyB_2D_i = yB_2B_1^{-1}.$
\item Compute $x.$
\end{enumerate}

\medskip
\paragraph{4) Stickel's key exchange protocol} \cite{Sti}. 

\medskip
Let $G$ be a nonabelian finite group and let  $g$ and $ f$ be two  non-commuting elements of $G.$ Let $k_0$ and $l_0$ be the  orders of $g$ and $f,$  respectively. It is assumed that $G, g, f, k_0, l_0$ are public.

\medskip
{\it Algorithm.} Alice picks two private positive numbers $k$ and $l,$ $1< k< k_0, 1 < l < l_0$,  and  publishes  $g^kf^l.$ Bob picks two private positive numbers 
$r$ and $s,$ $1 < r < k_0, 1 < s < l_0$, and publishes  $g^rf^s.$ 

\medskip
{\it Key establishment.} Alice computes the element $K_A = g^k(g^rf^s)f^l = g^{k+r}f^{l+s}.$ Bob computes the element $K_B = g^r(g^kf^l)f^s = g^{k+r}f^{l+s}.$ The shared key is  $K = K_A = K_B = g^{k+r}f^{l+s}.$

\medskip
{\it  Cryptanalysis.}  If $G \leq Mat_n(\MA)$ (as is the case in \cite{Sti})  then by Claim 3 there exists an algorithm that given the public data above finds the shared key $K$ in  time polynomial in $n, \dim_\F(\MA), k_0, l_0$ and the sizes of $g$ and $f$ (we assume that the field $\F$ is fixed).

\medskip
{\it Remark.}  A similar cryptanalysis applies when $G$ is an arbitrary (not necessary finite) linear group.

\medskip
\paragraph{5) \'{A}lvarez, Martinez' et. al.  key exchange protocol} \cite{AMVZ}, \cite{ATVZ1},  \cite{ATVZ2}. 

\medskip
Given a prime number $p$ and two positive numbers $n, m$  Alice and Bob choose two matrices 
$M_i = \left( \begin{array}{cc}A_i & X_i\\
0 & B_i\end{array} \right),$ for $i=1,2,$ respectively. Here, $A_i \in $ GL$_n(\mathbb{F}_p),$ $B_i \in $ GL$_m(\mathbb{F}_p), $  $X_i \in $ M$_{n \times m}(\mathbb{F}_p).$ 
Let $|M_i |= m_i,$ for $i=1,2,$ be the orders of these matrices, respectively. For any positive number $t$ one has $M_i^t =  \left( \begin{array}{cc}
A_i^t & X_i^{(t)}\\
0& B_i^t\\
\end{array}\right), i = 1,2.$

\medskip
{\it Algorithm.}

Alice picks two private positive numbers $k_i, 1\leq  k_i \leq m_i-1, i = 1,2,$ 
and  publishes  $C = M_1^{k_1}M_2^{k_2}= \left( \begin{array}{cc}
A_C & X_C\\
0& B_C\\
\end{array}\right).$ 

Bob picks two private positive numbers 
$l_i,$  $1 \leq l_i \leq  m_i-1, i = 1,2,$ and publishes $D = M_1^{l_1}M_2^{l_2} =    \left( \begin{array}{cc}
A_D & X_D\\
0& B_D\\
\end{array}\right).$

\medskip
{\it Key establishment.}

Alice computes  $K_A = A_1^{k_1}A_{D}X_2^{(k_2)} + A_1^{k_1}X_DB_2^{k_2} + X_1^{(k_1)}B_DB_2^{k_2}.$

Bob computes $K_B = A_1^{l_1}A_CX_2^{(l_2)} + A_1^{l_1}X_CB_2^{l_2} + X_1^{(l_1)}B_CB_2^{l_2}.$

The shared key is: $K = K_A = K_B .$

It was noted in \cite{VPD} that $K$ is the $(1,2)$ entry of $M_1^{k_1+l_1}M_2^{k_2+l_2}.$

\medskip
{\it  Cryptanalysis.} By Claim 3 there exists an algorithm that given the public data above finds the shared key $K$ in  time polynomial in $n, m, m_1, m_2$ and the sizes of the matrices $M_1$ and $M_2$ (we assume that the field $\F_p$ is fixed).

\medskip
\paragraph{6) Shpilrain-Ushakov's key exchange protocol} \cite{SU4} (see also \cite{MSU1}).  

Here we describe the general (non-twisted) version of the protocol and show that if the platform group is linear then there is an efficient attack to recover the shared key.  There is also a twisted version of the protocol, to which our  cryptanalysis applies  as well, so we omit it here and we leave the details to the reader. Notice, that in the original paper \cite{SU4} the suggested platform is the Thompson group, which is non-linear, so our attack does not apply here.
 
\medskip
Let $G\leq $ M$_n(\mathbf{A})$ be a group (or a submonoid) and let $g$ be an element of $G.$ Let $A$ and $B$ be two public finitely generated subgroups (or submonoids) of $G$ commuting element-wise. 

\medskip
{\it Algorithm.} Alice picks private elements $a, a' \in A$ and publishes  the element $aga'.$ Bob picks private elements $b, b' \in B$  and  publishes the element $bgb'.$ 

\medskip
{\it Key establishment.}
Alice computes  $K_A = abgb'a'.$ Bob computes  $K_B = baga'b' = abgb'a'.$
The shared key is: $K=K_A = K_b = abgb'a'.$

\medskip
{\it  Cryptanalysis.}  By Claim 3 there exists an algorithm that given the public data above finds the shared key $K$ in  time polynomial in $n, \dim_\F(\MA),$ and the sizes of the fixed generating sets of $A$ and $B$ and the sizes of the elements $g, aga',$  and $bgb'$ (we assume that the field $\F$ is fixed).

\medskip
\paragraph{7) Romanczuk-Ustimenko key exchange protocol} \cite{RU}   

\medskip

Let $G = GL_n(\F)$,  where $ \F$ is a finite field.  
Suppose  $C, D \in G$  be two commuting matrices. Fix a vector $g \in \F^n$. All this data is  public. 

\medskip
{\it Algorithm.}  Alice picks a polynomial $P =P(C, D) \in  \F[x, y]$  and publishes  the vector $ gP$.   Bob picks a polynomial $Q = Q(C,D)\in \F[x, y]$  and publishes the vector 
$gQ$. 

\medskip
{\it Key establishment.}
 Alice computes $K_A = (gQ)P = gQP$ and Bob computes $K_B =  (gP)Q =gPQ.$  The shared key is the vector    $   K = K_A = K_B.$ 
 
 \medskip
{\it  Cryptanalysis.}  By Claim 1 there exists an algorithm that given the public data above finds the shared key $K$ in  time polynomial in $n$ and the sizes of $C,D,P,$ and $Q$ (we assume that the field $\F$ is fixed).

\medskip 
{\it Remark.}
Another attack on this protocol, also based on linear algebra, was proposed earlier  by Blackburn at. al. in \cite{BCM}. The main idea of the attack is as follows.

Suppose an adversary Eve knows $g, gP,gQ, C$, and $D$.  Let $X$  be any matrix such that $X$ commutes with $C$ and with $D,$ 
and such that $gQ = gX.$  To find such an $X$ it suffices to solve the corresponding system of linear equations. Now 
Eve can compute the shared key as 
$(gP)X  = gXP  = gQP = K.$

\subsection{Protocols using automorphisms of groups}

\paragraph{ 8) Mahalanobis' key exchange protocol 1} \cite{Mah}

\medskip
Let $G$ be a group and $g \in G$.   Suppose $U, W$ be two finite subsets of $Aut(G)$ commuting element-wise. Denote by $\Phi$ and $\Psi$ the subgroups in $Aut(G)$ generated by $U$ and $W$ respectively.  

\medskip
{\it Algorithm.} 
Alice picks $\phi \in  \Phi $ and publishes $\phi (g).$ Bob picks $\psi \in \Psi $ and publishes  $\psi (g).$

 \medskip
{\it Key establishment.}  Alice computes $K_A = \phi (\psi (g)).$ Bob computes $K_B = \psi (\phi (g)) = \phi (\psi (g)).$
The shared key is $K = K_A = K_B =  \phi (\psi (g)).$

\medskip
{\it Cryptanalysis.}
If $G$ is such that $Hol(G)$ is linear (a subgroup of $Mat_n(\MA)$)   then by Claim 4 there exists an algorithm that given the public data above finds the shared key $K$ in  time polynomial in $n, \dim_\F(\MA)$ and the sizes of $g$ and the elements in $U,W$  (we assume that the field $\F$ is fixed).

In the original paper  \cite{Mah} the author suggested a (finitely generated) non-abelian nilpotent group $G$ as the platform group.  It is known (see, for example, \cite{Mer,Segal,LR}) that the holomorph $Hol(G)$ of every polycyclic group, in particular, every  finitely generated nilpotent group,  admits a faithful matrix representation. Hence the analysis above holds. 

Observe, that  the efficacy of the attack depends on the size of the linear representation of the $Hol(G)$.  Not much is known about the dimensions  of minimal  faithful representations of the holomorphs of nilpotent groups. 

  \medskip
\paragraph{9)  Mahalanobis' key exchange protocol 2} \cite{Mah}

\medskip
We assume the notation above. 

\medskip
{\it Algorithm.} Alice picks $\phi \in  \Phi$ and sends  $g^\phi $ to  Bob.   Bob  picks $\psi \in \Psi $ and  sends $(g^\phi)^\psi  = g^{\phi\psi}$ back  to Alice. Alice computes $\phi^{-1} $ and  gets $g^\psi = g^{\phi \psi \phi^{-1}}$.  Then Alice picks another automorphism $\xi \in \Phi$ and sends $(g^\psi)^\xi  = g^{\psi\xi}$ to Bob.

 \medskip
{\it Key establishment.} Bob computes $\psi^{-1}$ and  gets $((g^\psi)^\xi)^{\psi^{-1} }= g^\xi$ which is his session key. 

\bigskip
{\it Cryptanalysis.} Similar to the case above.

Indeed, assume $Hol(G) \leq Mat_n(\MA)$. Put $v = g^{\phi\psi}$. As in Claim 1 one can find a basis of $Sp(v^\Psi)$, say $v^{c_1}, \ldots, v^{c_t}$ . Notice that $g^{\psi\xi} \in Sp(v^\Psi)$ so one can decompose

\begin{equation}
\label{eq:gamo1}
g^{\psi\xi} = \sum_{i=1}^t \alpha_i v^{c_i}          =  ( \sum_{i=1}^t \alpha_i (g^\phi)^{c_i})^\psi , \  \textrm{for} \  \alpha_i \in \mathbb{F}. 
\end{equation}
Hence, $ (g^{\xi})^\psi  = ( \sum_{i=1}^t \alpha_i (g^\phi)^{c_i})^\psi $, so we derive

\begin{equation}
\label{eq:gmo1}
g^{\xi} =   \sum_{i=1}^t \alpha_i (g^\phi)^{c_i}. 
\end{equation}

\medskip
\paragraph{10)  Habeeb, Kahrobaei, Koupparis and Shpilrain's  key exchange protocol } \cite{HKKS}

\medskip
Let $G$ be a (semi)group, and Aut$(G)$ be the automorphism group of $G.$ Let 
$H(G)$ be the holomorph of $G.$  Fix an element  $g \in G$ and an automorphism $\phi \in $ Aut$(G).$ All this data are public. 

For this paragraph we  write for $g \in G$ and $\mu \in $ Aut$(G)$  the image $\mu (g)$ instead of $g^{\mu}.$  

\medskip
{\it Algorithm.} Alice picks a private number  $m \in \mathbb{N}.$  Then she computes 
$(\phi , g)^m = (\phi^m, \phi ^{m-1}(g) \cdot ... \cdot \phi^2(g) \cdot \phi (g) \cdot g)$
and sends only the second component $a_m = \phi ^{m-1}(g) \cdot ... \cdot \phi^2(g) \cdot \phi (g) \cdot g$  of this pair to Bob. 

 Bob picks a private $n \in \mathbb{N}.$  Then he computes 
$(\phi , g )^n = (\phi^n, \phi ^{n-1}(g) \cdot ... \cdot \phi^2(g) \cdot \phi (g) \cdot g)$
and sends only the second component $a_n = \phi ^{n-1}(g) \cdot ... \cdot \phi^2(g) \cdot \phi (g) \cdot g$  of this pair to Alice. 

\medskip
{\it  Key establishment.}  
Alice computes 
$(\ast , a_n)\cdot (\phi^m,a_m) = (\ast \cdot \phi^m, \phi^m (a_n)\cdot a_m ) = (\ast \cdot \phi^m, K_A).$ Note that she does not actually "compute" $\ast \cdot \phi^m.$

Bob computes  $(\ast \ast ,a_m)\cdot (\phi^n,a_n) = (\ast \ast \cdot \phi^n, \phi^n (a_m)\cdot a_n) = (\ast \ast \cdot \phi^n, K_B).$ Note that he does not actually "compute" $\ast \ast \cdot \phi^n.$

 The shared key is $K = K_A = K_B = a_{m+n}.$ 

\medskip
{\it  Cryptanalysis.} Let   $G \leq \MA ,$ where $\MA$ is a finite dimensional associative algebra over a field $\F$. Assume that the automorphism $\phi $
 is extended to an automorphism of the underlying vector space of $\MA$.  

Using Gauss elimination we can effectively find  a maximal linearly independent subset $L$ of the set $\{a_0, a_1, ..., a_k, ... \},$   where $a_0 = g$ and $a_k = \phi^{k-1}(g) \cdot ... \cdot \phi (g) \cdot g$  for $k \geq 1.$ Indeed, suppose that $\{a_0, ..., a_k\}$ is linearly independent set but 
$a_{k+1}$ can be presented as a linear combination of the form

$$a_{k+1} = \sum_{i=0}^k\lambda_ia_i \  \textrm{for} \   \lambda_i \in \F.
$$

Suppose by induction that $a_{k+j}$ can be presented as above for every $j \leq t-1.$ In particular

$$a_{k+t-1} = \sum_{i=0}^k\mu_i a_i  \  \textrm{for some} \  \mu_i \in \F.$$

Then

$$a_{k+t} = \phi (a_{k + t-1}) \cdot g =    \sum_{i=0}^k\mu_i \phi (a_i) \cdot g =$$
$$\sum_{i= 0}^{k}\mu_i a_{i+1} =  \mu_k\lambda_0a_0 + \sum_{i=0}^{k-1}(\mu_i + \mu_k\lambda_{i+1})a_{i+1}.$$
Thus $L = \{a_0, ..., a_k\}.$  

In particular, we can effectively compute 

$$a_n = \sum_{i=0}^k \eta_i a_i \   \textrm{for some} \  \eta_i \in \F.$$

Then 

$$a_{m+n} = \phi^m(a_n)\cdot  a_m = \sum_{i=0}^k \eta_i \phi^m(a_i) \cdot a_m=$$

$$= \sum_{i=0}^k \eta_i \phi^i(a_m) \cdot a_i.$$

Note that all data on the right hand side is known now. Thus we get the shared key $K = a_{m+n}.$

In the original version of this cryptosystem  \cite{HKKS}    $G$ was proposed to be the  semigroup of $3 \times 3$ matrices over the group algebra $\F_7[\mathbb{A}_5]$,
where $\mathbb{A}_5$ is the alternating group on $5$ elements. The authors of \cite{HKKS}  used an extension of the semigroup $G$ by an inner automorphism which is conjugation by a matrix $H \in $ GL$_3(\F_7[\mathbb{A}_5]).$  Therefore, in this case   there is a polynomial time algorithm to find the shared key $K$ from the public data.

\end{document}